\documentclass[12pt]{article}

\usepackage{tikz,color,url}
\usepackage{amsmath,amssymb,latexsym,enumerate,amsthm}

\newtheorem{theorem}{Theorem}[section]

\newtheorem{lemma}[theorem]{Lemma}
\newtheorem{proposition}[theorem]{Proposition}

\newtheorem{corollary}[theorem]{Corollary}
\newtheorem{remark}[theorem]{Remark}

\newcommand{\diam}{{\rm diam}}

\newcommand{\n}{{\rm n}}

\begin{document}

\title{Total mutual-visibility in graphs with emphasis on lexicographic and Cartesian products}

\author{
Dorota Kuziak $^{a}$\thanks{Email: \texttt{dorota.kuziak@uca.es}}
\and
Juan A. Rodr\'iguez-Vel\'azquez $^{b}$\thanks{Email: \texttt{juanalberto.rodriguez@urv.cat}}
}

\date{}

\maketitle
\begin{center}
	$^a$ Departamento de Estad\'istica e Investigaci\'on Operativa, Universidad de C\'adiz, Spain \\
	\medskip

	$^b$ Departament d'Enginyeria Inform\`atica i Matem\`atiques, Universitat Rovira i Virgili, Spain
	\medskip
\end{center}

\begin{abstract}
Given a connected graph $G$, the total mutual-visibility number of $G$, denoted $\mu_t(G)$, is the cardinality of a largest set $S\subseteq V(G)$ such that for every pair of vertices $x,y\in V(G)$ there is a shortest $x,y$-path whose interior vertices are not contained in $S$. Several combinatorial properties, including bounds and closed formulae, for $\mu_t(G)$ are given in this article.
Specifically, we give several bounds for $\mu_t(G)$ in terms of the diameter, order and/or connected domination number of $G$ and show characterizations of the graphs achieving the limit values of some of these bounds. We also consider those vertices of a graph $G$ that either belong to every total mutual-visibility set of $G$ or does not belong to any of such sets, and deduce some consequences of these results. We determine the exact value of the total mutual-visibility number of lexicographic products in terms of the orders of the factors, and the total mutual-visibility number of the first factor in the product. Finally, we give some bounds and closed formulae for the total mutual-visibility number of Cartesian product graphs.
\end{abstract}


\noindent
{\bf AMS Subj.\ Class.\ (2020)}: 05C12, 05C76

\section{Introduction}
\label{sec:intro}

Vertex visibility in networks is a topic that has motivated a significant number of investigations in the last few years. These investigations have been taken into account from two different points of view. In one hand, the research has been conducted in practical problems appearing in the area of computer science, through the study of some robot navigation models, mainly focused on the visibility required for some robots to move around a network in order to avoid collisions. For some examples of researches on this topic we suggest for instance \cite{aljohani-2018a,bhagat-2020,diluna-2017}. In a second hand, a theoretical  point of view has been considered, where the investigation mainly concentrates the attention into finding combinatorial properties of ``visible'' sets in networks. The first ideas on this direction were presented in \cite{diStefano-2022}, where the concept of mutual-visibility number in graphs was introduced. This latter style of theoretical study was extended in \cite{cicerone-2023,cicerone-2022+}.

Let $G = (V(G), E(G))$ be a connected and undirected graph, $X\subseteq V(G)$, and $x,y\in X$. If there exists a shortest $x,y$-path (also called geodesic) whose internal vertices are all not in $X$, then $x$ and $y$ are $X$-\emph{visible}. The set $X$ is a \emph{mutual-visibility set} of $G$, if every two vertices $x$ and $y$ of $X$ are $X$-visible. The cardinality of the largest mutual-visibility set of $G$ is the \emph{mutual-visibility number} of $G$ denoted by $\mu(G)$.

The topic of mutual-visibility in graphs is closely related to the general position problem in graphs, which was formally, independently and recently defined in \cite{klavzar-2018,ullas-2016}, although its notion is already known from previous investigations, like for instance \cite{korner-1995}, where the concept was considered only for hypercubes. A general position set in a graph $G$ can be understood as a mutual-visibility set in $G$ in which any two vertices of such set are ``visible'' not only through at least one shortest path but through every possible shortest path between the two vertices. The general position problem has been intensively studied in the last 5 years, and by now there are many ongoing investigations on this topic and its variations. For some significant cases, we suggest some of the most recent ones \cite{ghorbani-2021,klavzar-2021-b,klavzar-2021,patkos-2020,tian-2021a,tian-2021}.

In order to give more insight into the mutual-visibility number of strong product graphs, Cicerone \emph{et al.} \cite{cicerone-2022+} introduced the notion of total mutual-visibility as a natural extension of the mutual-visibility, which can also be seen in a computer setting navigation model, where not only the robots are required to be ``visible'' with respect to themselves, but also the remaining nodes of the networks have a similar property among them, and with respect to the navigating robots. This setting is clearly more restrictive, but it surprisingly turns to become very useful while considering some networks having some Cartesian properties in the vertex set, namely, that ones of product-like structures, when a product is understood in the sense of the four classical graph products as defined in the book \cite{hammack-2011}.

As it happens, the concept of total mutual-visibility might be also of independent interest, as already pointed out in \cite{cicerone-2022+}, since the visibility is extended to all vertices in the graph, not only for the vertices from mutual-visibility set. This is clearly a property of independent interest, and its study is worthy of being continued. This was indeed already done in \cite{tian-2022+}, and it is our goal to continue finding more contributions on this regard. Formally, $X\subseteq V(G)$ is a \emph{total mutual-visibility set} of $G$, if every two vertices $x$ and $y$ of $G$ are $X$-visible. A largest total mutual-visibility set of $G$ is a \emph{$\mu_t(G)$-set}, its cardinality is the \emph{total mutual-visibility number} of $G$ denoted by $\mu_t(G)$.

We now give some basic terminology and basic definitions that shall be used through our whole exposition. Clearly, we continue considering here only connected and undirected graphs. Given a graph $G$ and two vertices $x,y\in V(G)$, the \emph{distance} $d_G(x,y)$ between $x$ and $y$ in $G$ is the length of a shortest $x,y$-path. The \emph{diameter} $\diam(G)$ of $G$ is the largest distance between pairs of vertices of $G$. The subgraph of $G$ induced by $S\subseteq V(G)$ will be denoted by $G[S]$ and the complement of a graph $G$ is $\overline{G}$. A subgraph $H$ of $G$ is \emph{convex} if for each two vertices $x,y\in V(H)$, all shortest $x,y$-paths in $G$ lie completely in $H$. As usual, the \emph{domination number} of $G$ is denoted by $\gamma(G)$, which is the cardinality of a smallest set such that any vertex not in the set is adjacent to at least one vertex of such set. By $n_1(G)$ we denote the number of vertices of degree one in $G$, also known as the number of \emph{leaves} of $G$, when $G$ is a tree. We next continue with some extra information (basic results) that we would need for our purposes.

We first recall that there exist graphs where $X=\varnothing$ is the only $\mu_t(G)$-set, like the case of cycles of order at least $5$, and graphs where  $X=V(G)$ is the only $\mu_t(G)$-set, like the case of complete graphs. Thus, for any graph $G$, $0 \le \mu_t(G)\le n(G)$, where $n(G)$ denotes the order of $G$. Also, the following observation from \cite{tian-2022+} is of interest.

\begin{proposition}{\em \cite{tian-2022+}}
If $X\subseteq V(G)$ is a total mutual-visibility set of a graph $G$ and $Y\subseteq  X$, then $Y$ is also a total mutual-visibility set of $G$.
\end{proposition}

To close this section we present the plan of our article. In Section \ref{sec:tmv}, several bounds for $\mu_t(G)$ in terms of the diameter, order and/or connected domination number of $G$ are given. We also show characterizations of the graphs achieving the limit values of some of these bounds, and present some consequences that gives the exact value of $\mu_t(G)$ when $G$ is a join or a corona graph. Section \ref{Section:Compulsory and forbidden vertices} is dedicated to consider those vertices of a graph $G$ that either belong to every total mutual-visibility set of $G$ or does not belong to any of such sets. In Section \ref{sec:lexicographic-product} we consider the lexicographic product of graphs $G$ and $H$, and compute the exact value of its total mutual-visibility number in terms of the orders of $G$ and $H$ and the total mutual-visibility number of $G$. In Section \ref{sec:tmv-Cartesian}, we give some bounds and closed formulae for the total mutual-visibility number of Cartesian product graphs. We determine such value for some specific families of graphs that are generalizing some other results recently presented in \cite{tian-2022+}. Finally, in the concluding section some open problems and directions for further investigation are indicated.

\section{General bounds and consequences}
\label{sec:tmv}

Our first contribution describes a relationship between the total mutual-visibility number and the diameter of a graph.

\begin{theorem}
\label{MainBound}
For any connected graph $G$,
$$0\le \mu_t(G)\le \n(G)-\diam(G)+1.$$
\end{theorem}

\begin{proof}
By definition of total mutual-visibility number, $\mu_t(G)\ge 0$. Now, let $X$ be a $\mu_t(G)$-set. If $u,v\in V(G)$ are two diametral vertices, then there exists a diametral path $u=x_0,\dots, x_k=v$ such that $\{x_{1},\dots, x_{k-1}\}\cap X=\varnothing$. Therefore,
$\mu_t(G)=|X|\le \n(G)-(k-1)=\n(G)-\diam(G)+1.$
\end{proof}

All graphs with $\mu_t(G)=0$ were characterized in \cite{tian-2022+}. Next we consider the case of graphs with $\mu_t(G)= \n(G)-\diam(G)+1.$

\begin{proposition}
\label{BoundDiameterCharact}
Given a  connected graph $G$ of order $\n(G)\ge 2$, the following statements are equivalent.
\begin{enumerate}[{\rm (i)}]
\item $\mu_t(G)= \n(G)-\diam(G)+1.$
\item There exists a diametral path $x_0,\dots, x_k$ such that for every pair $u,v$ of vertices of $G$ there exists a shortest path $u=y_0,\dots, y_{k'}=v$ such that $\{y_1,\dots, y_{k'-1}\}\subseteq \{x_1,\dots,x_{k-1}\}$.
\end{enumerate}
\end{proposition}

\begin{proof}
First, we assume that (i) holds. Let $X$ be a $\mu_t(G)$set. Let $x,y\in V(G)$ be two diametral vertices. Since $x$ and $y$ are $X$-visible, there exists a path $P$ between $x$ and $y$ whose set $Y$ of internal vertices satisfies $X\cap Y=\varnothing$. Hence, $Y\subseteq V(G)\setminus X$, and so from (i) we deduce that
$$\diam(G)-1=|Y|\le |V(G)\setminus X|=\diam(G)-1.$$ Therefore, $X=V(G)\setminus Y$, which implies that $P$ satisfies (ii).

Conversely, if $W$ is the set of internal vertices of a diametral path $P$ satisfying (ii), then  $V(G)\setminus W$ is a total mutual-visibility set, which implies that $\n(G)-\diam(G)-1=|V(G)\setminus W|\le \mu_t(G)$. In such a case, by Theorem~\ref{MainBound} we deduce that (i) holds.
\end{proof}

From Proposition~\ref{BoundDiameterCharact} we deduce the following result which characterizes the graphs with large values of total mutual-visibility number.

\begin{corollary}
\label{cor:tmv=n}
Given a graph $G$, the following statements hold.
\begin{itemize}
\item [{\rm (i)}] $\mu_t(G) = \n(G)$ if and only if $G$ is a complete graph.
\item [{\rm (ii)}] $\mu_t(G) = \n(G)-1$ if and only if $G$ is a non-complete graph with $\gamma(G)=1$.
\end{itemize}
\end{corollary}

Now, we establish an interesting connection between the total mutual-visibility number and the connected domination number, denoted by $\gamma_c(G)$, which represents the minimum cardinality among all dominating sets of $G$ whose induced subgraphs are connected. A smallest connected dominating set of $G$ is a \emph{$\gamma_c(G)$-set}.

\begin{theorem}
\label{upperboundConnectedDomination}
For any connected non-complete graph $G$,
$$\mu_t(G)\le \n(G)-\gamma_c(G).$$
\end{theorem}

\begin{proof}
Let $X$ be a $\mu_t(G)$-set and $X'=V(G)\setminus X$. Since $G$ is a non-complete graph, by  Theorem~\ref{MainBound} and Corollary~\ref{cor:tmv=n}, we deduce that $X'\ne \varnothing$. If there exists $x\in X$ such that $N(x)\cap X'=\varnothing$, then for every $x'\in X'$ we have that $x$ and $x'$ are not $X$-visible, which is a contradiction. Hence, $X'$ is a dominating set. Now, if $x'$ and $x''$ are vertices of two different components of the subgraph of $G$ induced by $X'$, then $x'$ and $x''$ are not $X$-visible, which is a contradiction again. Thus, $X'$ is a connected dominating set of $G$. Therefore,
$\n(G)=|X|+|X'|\ge\mu_t(G)+\gamma_c(G)$, as required.
\end{proof}

It is easy to construct examples of graphs achieving the bound above. In particular, as we will show in Theorem~\ref{th:tmv-corona}, the bound is achieved for connected corona graphs.

\begin{corollary}
\label{CharacterizationEqDiaameter}
 Let $G$ be a non-complete graph. If $\mu_t(G)= \n(G)-\diam(G)+1$, then $\gamma_c(G)=\diam(G)-1$.
\end{corollary}

\begin{proof}
Since any path between two diametral vertices has $\diam(G) - 1$ internal
vertices, it is clear that $\gamma_c(G)\ge  \diam(G) - 1$.
Hence, if $\mu_t(G)= \n(G)-\diam(G)+1$, then by Theorem~\ref{upperboundConnectedDomination} we have
$$\n(G)-\diam(G)+1=\mu_t(G)\le  \n(G)-\gamma_c(G) \le  \n(G)-\diam(G)+1,$$
which implies that $\gamma_c(G)=\diam(G) - 1$.
\end{proof}

The converse of Corollary~\ref{CharacterizationEqDiaameter} does not hold.  For instance, if $G$ is the graph shown in Figure~\ref{FigConnectedDomination}, then $\gamma_c(G)=5=\diam(G)-1$, while $\mu_t(G)=2<7=\n(G)-\diam(G)+1$.

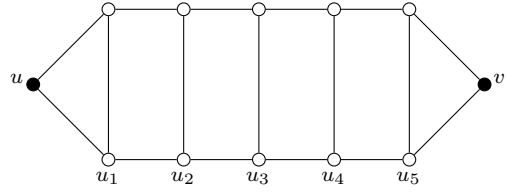
\begin{figure}[h]
\begin{center}
\begin{tikzpicture}[transform shape, inner sep = .6mm]

\node [draw=black, shape=circle, fill=black] (a) at  (-3,0) {};
\node [draw=black, shape=circle, fill=white] (b) at  (-2,-1) {};
\node [draw=black, shape=circle, fill=white] (c) at  (-1,-1) {};
\node [draw=black, shape=circle, fill=white] (d) at  (0,-1) {};
\node [draw=black, shape=circle, fill=white] (e) at  (1,-1) {};
\node [draw=black, shape=circle, fill=white] (f) at  (2,-1) {};
\node [draw=black, shape=circle, fill=black] (g) at  (3,0) {};
\node [draw=black, shape=circle, fill=white] (h) at  (2,1) {};
\node [draw=black, shape=circle, fill=white] (i) at  (1,1) {};
\node [draw=black, shape=circle, fill=white] (j) at  (0,1) {};
\node [draw=black, shape=circle, fill=white] (k) at  (-1,1) {};
\node [draw=black, shape=circle, fill=white] (l) at  (-2,1) {};

\draw (a)--(b)--(c)--(d)--(e)--(f)--(g)--(h)--(i)--(j)--(k)--(l)--(a);
\draw (b)--(l); \draw (c)--(k); \draw (d)--(j);
\draw (e)--(i);
\draw (f)--(h);

\node [left] at (-3.05,0) {$^{u}$};
\node [below] at (-2,-1.1) {$^{u_1}$};
\node [below] at (-1,-1.1) {$^{u_2}$};
\node [below] at (0,-1.1) {$^{u_3}$};
\node [below] at (1,-1.1) {$^{u_4}$};
\node [below] at (2,-1.1) {$^{u_5}$};
\node [right] at (3.05,0) {$^{v}$};

\end{tikzpicture}\caption{$\{u_1,\dots, u_5\}$ is a $\gamma_c(G)$-set.} \label{FigConnectedDomination}
\end{center}
\end{figure}

\begin{proposition}\label{CharactBoundConnectedDominationCharact}
Given a  connected non-complete graph $G$, the following statements are equivalent.
\begin{enumerate}[{\rm (i)}]
\item $\mu_t(G)= \n(G)-\gamma_c(G).$
\item There exists a $\gamma_c(G)$-set $S$  such that for every pair $u,v$ of vertices of $G$ there exists a shortest path $u=y_0,\dots, y_{k'}=v$ such that $\{y_1,\dots, y_{k'-1}\}\subseteq S$.
\end{enumerate}
\end{proposition}

\begin{proof}
First, we assume that (i) holds. Let $X$ be a $\mu_t(G)$-set and $X'=V(G)\setminus X$. As we have shown in the proof of Theorem~\ref{upperboundConnectedDomination}, $X'$ is a connected dominating set, and so from (i) we deduce that
$$\gamma_c(G)\le |X'|= |V(G)\setminus X|=\n(G)-\mu_t(G) =\gamma_c(G).$$
Therefore, $X'$ is a $\gamma_c(G)$-set which satisfies (ii).

Conversely, if $S$ is a $\gamma_c(G)$-set satisfying (ii), then  $V(G)\setminus S$ is a total mutual-visible set, which implies that $\n(G)-\gamma_c(G)=|V(G)\setminus S|\le \mu_t(G).$ In such a case, by Theorem~\ref{upperboundConnectedDomination} we deduce that (i) holds.
\end{proof}

Notice that from Proposition~\ref{CharactBoundConnectedDominationCharact} we deduce the following result which was recently obtained in \cite{tian-2022+}.
\begin{corollary}{\rm \cite{tian-2022+}}
For any tree $T$ of order at least three, $\mu_t(T)=n_1(T)$.
\end{corollary}

Next we will apply Proposition~\ref{CharactBoundConnectedDominationCharact} to the cases of join and corona graphs.
The \emph{join graph} $G+H$ is defined as the graph obtained from the disjoint union of a copy of $G$ and a copy of $H$ by adding an edge between each vertex of $G$ and each vertex of $H$.

\begin{corollary}
\label{th:tmv-corona}
Let $G$ and $H$ be two non-simultaneously complete graphs.
\begin{itemize}
  \item [{\rm (i)}] If $\gamma(G)=1$, then  $\mu_t(G+H)= \n(G)+\n(H)-1.$
  \item [{\rm (ii)}] If $\gamma(G)\ne 1$ and $\gamma(H)\ne 1$, then $\mu_t(G + H)= \n(G)+\n(H)-2.$
\end{itemize}
\end{corollary}

\begin{proof}
Since every universal vertex of $G$ is a universal vertex of $G+H$, from Proposition~\ref{CharactBoundConnectedDominationCharact} or from Corollary~\ref{cor:tmv=n}, we deduce (i).

Now, if   $\gamma(G)\ne 1$ and $\gamma(H)\ne 1$, then $\gamma_c(G+H)=2$ and  we only need to observe that for any vertex $g\in V(G)$ and any vertex $h\in V(H)$, the set $\{g,h\}$ is a connected dominating set   of $G+H$ which satisfies Proposition~\ref{CharactBoundConnectedDominationCharact} (ii).
\end{proof}

Let $G$ and $H$ be graphs where $V(G) = \{u_1, \ldots,u_{\n(G)}\}$.  The \emph{corona product} $G\odot H$ is defined as the graph obtained from the disjoint union of a copy of $G$ and $\n(G)$ copies of $H$, denoted by $H_i$, $i\in \{1,2,\dots,\n(G)\}$. The product $G\odot H$ is then constructed by making $u_i$ adjacent to every vertex in $H_i$ for each $i\in \{1,2,\dots,\n(G)\}$.
Notice that the corona product $K_1\odot H$ is isomorphic to the join graph $K_1+H$.

\begin{corollary}
\label{th:tmv-corona}
For any connected graph $G$ and any graph $H$, $$\mu_t(G\odot H)= \n(G)\n(H).$$
\end{corollary}

\begin{proof}We only need to observe that $V(G)$ is a $\gamma_c(G\odot H)$-set which satisfies Proposition~\ref{CharactBoundConnectedDominationCharact} (ii).
%
%
\end{proof}

\section{Compulsory vertices and forbidden vertices in any $\mu_t(G)$-set}
\label{Section:Compulsory and forbidden vertices}

In order to give some additional results on the total mutual-visibility number of a graph, we need to introduce the following notation.
Given a graph $G$, we define the set $\mathcal{F}(G)\subseteq V(G)$ as the set of forbidden vertices in any $\mu_t(G)$-set, i.e., the set of vertices not belonging to any $\mu_t(G)$-set.

We also define $\mathcal{C}(G) $ as the set of compulsory vertices in any $\mu_t(G)$-set, i.e., the set of vertices belonging to every $\mu_t(G)$-set.

If $X$ is a $\mu_t(G)$-set, then by definition of $\mathcal{C}(G)$ and $\mathcal{F}(G)$  we have that $\mathcal{C}(G)\subseteq X\subseteq V(G)\setminus \mathcal{F}(G)$. Therefore, the following proposition follows.

\begin{proposition}
\label{prop:tmv-general-bounds}
If $G$ is a connected graph, then $|\mathcal{C}(G)| \le \mu_t(G)\le \n(G)-|\mathcal{F}(G)|$.
\end{proposition}

Although the following statement is also immediate, it is very useful.

\begin{proposition}
\label{general-bounds-equality}
Given a graph  $G$, the following statements are equivalent.
\begin{enumerate}[{\rm (i)}]
\item $\mu_t(G)=|\mathcal{C}(G)|$.
\item $\mu_t(G)= \n(G)-|\mathcal{F}(G)|$.
\item $|\mathcal{C}(G)|+|\mathcal{F}(G)|=\n(G)$.
\end{enumerate}
\end{proposition}

\begin{proof}
If (i) $\Leftrightarrow$ (ii) holds, then the other equivalences are deduced by Proposition~\ref{prop:tmv-general-bounds}. Therefore, we will limit ourselves to prove the equivalence (i) $\Leftrightarrow$ (ii).

Let $X$ be a $\mu_t(G)$-set. Since $\mathcal{C}(G)\subseteq X$, if $\mu_t(G)=|\mathcal{C}(G)|$, then  $\mathcal{C}(G)$ is the only $\mu_t(G)$-set. In such a case, $V(G)\setminus \mathcal{C}(G)=\mathcal{F}(G)$, and so $|\mathcal{F}(G)|= \n(G)-\mu_t(G)$, as required.

Conversely, if $\mu_t(G)= \n(G)-|\mathcal{F}(G)|$, then $V(G)\setminus \mathcal{F}(G)$ is the only $\mu_t(G)$-set, and this implies that  $V(G)\setminus \mathcal{F}(G)=\mathcal{C}(G)$. Therefore, $\mu_t(G)=\n(G)-|\mathcal{F}(G)|=|\mathcal{C}(G)|$.
\end{proof}

Figure~\ref{FigureN-Compulsory} shows an example of graph which illustrates Proposition~\ref{general-bounds-equality}.

\begin{figure}[h]
\begin{center}
\begin{tikzpicture}[transform shape, inner sep = .6mm]

\node [draw=black, shape=circle, fill=black] (a) at  (-3,0) {};
\node [draw=black, shape=circle, fill=white] (b) at  (-2,0) {};
\node [draw=black, shape=circle, fill=white] (c) at  (-1,-1) {};
\node [draw=black, shape=circle, fill=black] (d) at  (0,-2) {};
\node [draw=black, shape=circle, fill=white] (e) at  (1,-1) {};
\node [draw=black, shape=circle, fill=white] (f) at  (2,0) {};
\node [draw=black, shape=circle, fill=black] (g) at  (3,0) {};
\node [draw=black, shape=circle, fill=white] (h) at  (1,1) {};
\node [draw=black, shape=circle, fill=black] (i) at  (0,2) {};
\node [draw=black, shape=circle, fill=white] (j) at  (-1,1) {};
\node [draw=black, shape=circle, fill=white] (k) at  (0,0) {};

\draw (a)--(b)--(c)--(d)--(e)--(f)--(g);
\draw (f)--(h)--(i)--(j)--(j)--(b);
\draw (c)--(k)--(h);
\draw (e)--(k)--(j);

\node [above] at (-3,0.1) {$^{a}$};
\node [above] at (-2,0.1) {$^{b}$};
\node [below] at (-1,-1.1) {$^{c}$};
\node [below] at (0,-2.1) {$^{d}$};
\node [below] at (1,-1.1) {$^{e}$};
\node [above] at (2,0.1) {$^{f}$};
\node [above] at (3,0.1) {$^{g}$};
\node [above] at (1,1.1) {$^{h}$};
\node [above] at (0,2.1) {$^{i}$};
\node [above] at (-1,1.1) {$^{j}$};
\node [above] at (0,0.1) {$^{k}$};

\end{tikzpicture}\caption{$\mathcal{C}(G)=\{a,d,g,i\}$ is the only  $\mu_t(G)$-set.} \label{FigureN-Compulsory}
\end{center}
\end{figure}
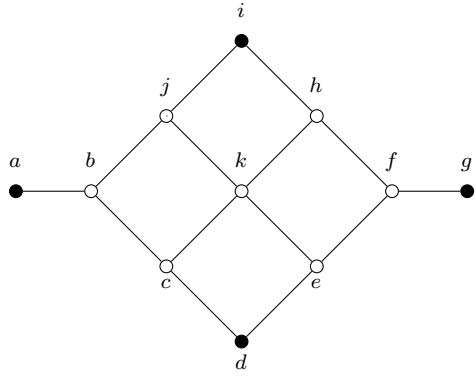

A vertex of a graph is a {\em simplicial} if the subgraph induced by its neighbors is a complete graph. Let $\mathcal{S}(G)$ be the set of simplicial vertices in $G$. Observe that $\mathcal{S}(G)\subseteq \mathcal{C}(G)$. Let $\mathcal{P}(G)$ be the subset of $V(G)$ such that $v\in \mathcal{P}(G)$   if and only if there exist two vertices $u,w\in V(G)$ such that $N_G[u]\cap N_G[w]=\{v\}$,
i.e., $v\in \mathcal{P}(G)$ if and only if $v$ is the middle vertex of a convex $P_3$ in $G$.  Obviously, $\mathcal{P}(G)\subseteq \mathcal{F}(G)$.

Notice that the problem of deciding if a vertex belongs to $\mathcal{S}(G)$, or to $\mathcal{P}(G)$, is algorithmically simple. Therefore, the following result is an important tool to estimate the value of $\mu_t(G)$.


\begin{proposition}
\label{pro:tmv-general-bounds-equality}
Given a connected graph $G$, the following statements hold.
\begin{enumerate}[{\rm (i)}]
\item $\mu_t(G) \ge |\mathcal{S}(G)|$.

\item If $\mu_t(G) = |\mathcal{S}(G)|$, then $\mu_t(G) = \n(G)-|\mathcal{P}(G)|$.

\item { \em \cite{tian-2022+}}  $\mu_t(G) \le \n(G)-|\mathcal{P}(G)| $.

\item If $\mu_t(G) = \n(G)-|\mathcal{P}(G)| $, then $\mu_t(G) = |\mathcal{C}(G)|$.
\end{enumerate}
\end{proposition}

\begin{proof}
From Propositions~\ref{prop:tmv-general-bounds} and \ref{general-bounds-equality} we deduce (i),   (iii) and (iv). We proceed to prove (ii). Since $\mathcal{S}(G)\subseteq \mathcal{C}(G)$, if $\mu_t(G)=|\mathcal{S}(G)|$, then the set $\mathcal{S}(G)$ of simplicial vertices is the only $\mu_t(G)$-set. If there exists a vertex $x\in V(G)\setminus (\mathcal{S}(G)\cup \mathcal{P}(G))$, then for every pair of non-adjacent vertices $y,z\in N_G(x)$, there exists $w\in V(G)\setminus \mathcal{S}(G)$ such that $y,z\in N_G(w)$, and so $\mathcal{S}(G)\cup \{x\}$ is a total mutual-visibility set, which is a contradiction. Therefore, $V(G)=\mathcal{S}(G)\cup \mathcal{P}(G)$, as required.
\end{proof}

The graph $G$ shown in Figure~\ref{FigConnectedDomination} is also an example which illustrates Proposition~\ref{pro:tmv-general-bounds-equality} (ii). The set $\mathcal{S}(G)$ is formed by the bold  vertices, and $\mathcal{P}(G)$ by the white ones. Moreover, as previously mentioned, $\mu_t(G)=2$.

As we can expect, the converse of Proposition~\ref{pro:tmv-general-bounds-equality} (ii) does not hold. For instance, if $G$ is the graph shown in Figure~\ref{FigureN-Compulsory}, then  $\mathcal{S}(G)=\{a,g\}\subseteq \{a,d,g,i\}=\mathcal{C}(G)$ and $\mathcal{P}(G)=V(G)\setminus \mathcal{C}(G)$. Therefore, $\mu_t(G) = \n(G)-|\mathcal{P}(G)|=|\mathcal{C}(G)|=4>2=|\mathcal{S}(G)|.$

Notice also that the converse of Proposition~\ref{pro:tmv-general-bounds-equality} (iv) does not hold. For instance, if $G$ is the graph shown in Figure~\ref{Figure-Forbiden}, then $\mathcal{P}(G)=\{b,c,e,f,h,j\}\subseteq \{b,c,e,f,h,j,k\}=\mathcal{F}(G)$. In this case, $\mu_t(G) = |\mathcal{C}(G)|=\n(G)-|\mathcal{F}(G)|=6<7=\n(G)-|\mathcal{P}(G)|.$

\begin{figure}[h]
\begin{center}
\begin{tikzpicture}[transform shape, inner sep = .6mm]

\node [draw=black, shape=circle, fill=black] (a) at  (-3,0) {};
\node [draw=black, shape=circle, fill=white] (b) at  (-2,0) {};
\node [draw=black, shape=circle, fill=white] (c) at  (-1,-1) {};
\node [draw=black, shape=circle, fill=black] (d) at  (0,-2) {};
\node [draw=black, shape=circle, fill=white] (e) at  (1,-1) {};
\node [draw=black, shape=circle, fill=white] (f) at  (2,0) {};
\node [draw=black, shape=circle, fill=black] (g) at  (3,0) {};
\node [draw=black, shape=circle, fill=white] (h) at  (1,1) {};
\node [draw=black, shape=circle, fill=black] (i) at  (0,2) {};
\node [draw=black, shape=circle, fill=white] (j) at  (-1,1) {};
\node [draw=black, shape=circle, fill=white] (k) at  (0,0) {};
\node [draw=black, shape=circle, fill=black] (l) at  (0,1) {};
\node [draw=black, shape=circle, fill=black] (m) at  (0,-1) {};
\draw (a)--(b)--(c)--(d)--(e)--(f)--(g);
\draw (f)--(h)--(i)--(j)--(j)--(b);
\draw (c)--(k)--(h);
\draw (c)--(l)--(h);
\draw (e)--(k)--(j);
\draw (e)--(m)--(j);

\node [above] at (-3,0.05) {$^{a}$};
\node [above] at (-2,0.05) {$^{b}$};
\node [below] at (-1,-1.1) {$^{c}$};
\node [below] at (0,-2.1) {$^{d}$};
\node [below] at (1,-1.1) {$^{e}$};
\node [above] at (2,0.1) {$^{f}$};
\node [above] at (3,0.1) {$^{g}$};
\node [above] at (1,1.05) {$^{h}$};
\node [above] at (0,2.02) {$^{i}$};
\node [above] at (-1,1.05) {$^{j}$};
\node [above] at (0,1.02) {$^{l}$};
\node [below] at (0,-1.06) {$^{m}$};
\node [above] at (0,0.06) {$^{k}$};

\end{tikzpicture}\caption{$\mathcal{C}(G)=\{a,d,g,i,l,m\}$ is the only  $\mu_t(G)$-set.} \label{Figure-Forbiden}
\end{center}
\end{figure}
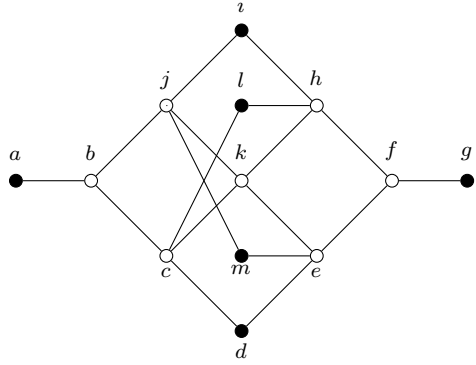

Since the vertex set of any non-complete block graph $G$ can be partitioned as $V(G)= \mathcal{P}(G)\cup \mathcal{S}(G)$, we deduce the following result.

\begin{corollary}
\label{cor:tmv-exact-values}
If $G$ is a block graph, then $\mu_t(G) = |\mathcal{S}(G)|=\n(G)-|\mathcal{P}(G)|$.
\end{corollary}

\section{The case of lexicographic product  graphs}
\label{sec:lexicographic-product}
The lexicographic product $G\circ H$ of two graphs $G$ and $H$ is a graph with vertex set $V(G\circ H)=V(G)\times V(H)$. Two vertices $(x,y)$ and $(x',y')$ are adjacent if $xx'\in E(G)$ or ($x=x'$ and $yy'\in E(H$)). The lexicographic product is a kind of generalization of join because $K_2\circ G\cong G + G$ for any graph $G$. If $S\subseteq V(G\circ H)$, then the projection $S_G$ of $S$ on $G$ is the set $\{g\in V(G): (g,h)\in S \textrm{ for some } h\in V(H)\}$. The projection $S_H$ of $S$ on $H$ is defined analogously.

Since the  concept of total mutual-visibility is defined for connected graphs, we should remember that a lexicographic product $G\circ H$ is connected if and only if $G$ is connected. Moreover, the relation between distances in a lexicographic product  graph  and in its factors can be presented as follows.

\begin{remark}{\em \cite{hammack-2011}}
If $G$ is a connected non-trivial graph, then the distance between two vertices  $(g,h)$ and $(g',h')$   of $G\circ H$ is given by
$$d_{G\circ H}((g,h),(g',h'))=\left\{\begin{array}{ll}

                            d_G(g,g')  & \mbox{if $g\ne g'$,} \\
                            [10pt]
                            \min\{d_H(h,h'),2\}  & \mbox{if $g=g'$.}
                            \end{array} \right.\\
$$
\end{remark}

For more on the product graphs see \cite{hammack-2011}.

\begin{lemma}
\label{lem:tmv-all-but-one}
Let $G$ be a graph with $\gamma(G)\ge 2$ and let $H$ be a graph. If $X$ is a $\mu_t(G\circ H)$-set, then $|\{u\}\times V(H)\cap X|\ge \n(H)-1$ for every vertex $u\in V(G)$.
\end{lemma}

\begin{proof}
Let $u\in V(G)$. By the maximality of the cardinality of $X$, if there exist two different vertices $v, v'\in V(H)$ such that $(u,v),(u,v')\notin X$, then $\{u'\}\times V(H)\subseteq X$ for every $u'\in N_G(u)$. Now, since $\gamma(G)\ge 2$, there exists $w\in V(G)\setminus N_G[u]$, and so  two vertices $(u,v)$ and $(w,v)$ are not $X$-visible, which is a contradiction. Therefore, the result follows.
\end{proof}

\begin{lemma}
\label{lem:tmv-neigh-out}
Let $G$ be a graph with $\gamma(G)\ge 2$. For any $\mu_t(G)$-set $X$ and any vertex $v\in V(G)$, $N_G(v)\cap (V(G)\setminus X)\ne \varnothing$.
\end{lemma}

\begin{proof}
Suppose that there exits a $\mu_t(G)$-set $X$, and a vertex $v\in V(G)$, such that $N_G(v)\subseteq X$. Since $\gamma(G)\ge 2$, there exists  $u\in V(G)\setminus N_G[v]$. Hence,  $u$ and $v$ are not $X$-visible, which is a contradiction.
\end{proof}

\begin{theorem}
Let $G$ be a connected graph with $\gamma(G)\ge 2$. For any graph $H$, $$\mu_t(G\circ H) = \n(G)(\n(H)-1) + \mu_t(G).$$
\end{theorem}

\begin{proof}
Let $S$ be a $\mu_t(G)$-set and let $v\in V(H)$. We proceed to show that $X = S\times V(H) \cup (V(G)\setminus S)\times (V(H)\setminus \{v\})$ is a total mutual-visibility set. We differentiate two cases for $(u,v),(u',v')\in V(G\circ H)$.

\noindent Case 1: $u=u'$. By Lemma \ref{lem:tmv-neigh-out} there exists a vertex $w\in N_G(u)\cap (V(G)\setminus S)$. Thus, the  vertices $(u,v)$ and $(u',v')$ are $X$-visible.

\noindent Case 2: $u\ne u'$. Let $u=u_0,u_1,\dots, u_k=u'$ be a shortest path in $G$ such that $\{u_1,\dots,u_{k-1}\}\cap S =\varnothing$. Since $(u,v)=(u_0,v),(u_1,v),\dots,(u_k,v)=(u',v')$ is a shortest path in $G\circ H$ and $\{u_1,\dots, u_{k-1}\}\times \{v\}\cap X = \varnothing$, the vertices $(u,v)$ and $(u',v')$ are $X$-visible.

Therefore, $X$ is a total mutual-visibility set, and so $$\mu_t(G\circ H)\ge |X| = \n(G)(\n(H)-1) + \mu_t(G).$$

Let $W$ be a $\mu_t(G\circ H)$-set and $W'=V(G\circ H)\setminus W$. Let $W'_G$ be the projection of $W'$ in $G$. We proceed to show that $W_G=V(G)\setminus W'_G$ is a total mutual-visibility set of $G$. For every pair of different vetices $x,x'\in V(G)$ and $y\in V(H)$, there exists a shortest path $(x,y)=(x_0,y_0),\dots, (x_k,y_k)=(x',y)$ such that $(x_1,y_1),\dots, (x_{k-1},y_{k-1})\notin W$. Thus, $x=x_0,\dots, x_k=x'$ is a shortest path in $G$ and $x_1,\dots, x_{k-1}\notin W_G$. Hence, the pair of vertices $x$ and $x'$ are $W_G$-visible and so $|W_G|\le \mu_t(G)$. Therefore, by Lemma \ref{lem:tmv-all-but-one} we have that
\begin{align*}
  \mu_t(G\circ H) & =|W| \\
                  & =|W_G|\cdot\n(H)+|W'_G|(\n(H)-1) \\
                  & =\n(G)(\n(H)-1)+|W_G| \\
                  & \le \n(G)(\n(H)-1) + \mu_t(G),
\end{align*}
as required.
\end{proof}

\begin{theorem}
Let $G$ be a graph with $\gamma(G)=1$.
\begin{itemize}
  \item [{\rm (i)}] If $H$ is a non-complete graph with $\gamma(H)=1$, then $$\mu_t(G\circ H)= \n(G)\n(H)-1.$$
  \item [{\rm (ii)}] If $H$ is a graph with $\gamma(H)\ge 2$, then $$\mu_t(G\circ H)= \n(G)\n(H)-2.$$
\end{itemize}
\end{theorem}

\begin{proof}
Let $u$ be a universal vertex of $G$.
If $v$ is a a universal vertex of $H$, then $(u,v)$ is a universal vertex of $G\circ H$, and so by Corollary~\ref{cor:tmv=n} we conclude that (i) follows.

In order to prove (ii), we take  an arbitrary vertex $u'\in V(G)\setminus \{u\}$ and an arbitrary vertex $v\in V(H)$, and since $u$ is a universal vertex of $G$, it is readily seen that $ X=V(G\circ H)\setminus \{(u,v),(u',v)\}$ is a total mutual-visibility set of $G\circ H$. Hence,  $\mu_t(G\circ H)\ge |X| = \n(G)\n(H)-2$. Now, if $\gamma(H)\ge 2$, then $\gamma(G\circ H)\ge 2$ and, as a result, Corollary~\ref{cor:tmv=n} leads to $\mu_t(G\circ H) \le \n(G)\n(H)-2$, which completes the proof of (ii).
\end{proof}

\section{The case of Cartesian product graphs}
\label{sec:tmv-Cartesian}

Let $G$ and $H$ be two graphs. The Cartesian product of $G$ and $H$   is the graph $G\Box H$ with $V(G\Box H)=V(G)\times V(H)$, where two vertices $(x,y)$ and $(x',y')$ are adjacent if and only if either $x=x'$ and $yy'\in E(H)$, or $xx'\in E(G)$ and $y=y'$. A Cartesian product graphs is connected if and only if both of its factors are connected. The distance between $(x,y)$ and $(x',y')$ in $G\Box H$ is given by
$$d_{G\Box H}((x,y),(x',y'))=d_G(x,x')+d_H(y,y').$$ For more information on structure and properties of the Cartesian product of graphs we refer the reader to \cite{hammack-2011}.

We recall that the general position and the mutual-visibility problems were investigated recently, for instance in \cite{ghorbani-2021, klavzar-2021, tian-2021} and \cite{cicerone-2023}, respectively. Moreover, Tian and Klav\v{z}ar \cite{tian-2022+} most recently studied the total mutual-visibility number of Cartesian product graphs. In the referred work, the authors gave general bounds on the total mutual-visibility number
of Cartesian product graphs, and they also obtain closed formulas on this novel parameter for specific families of Cartesian product graphs. To continue our exposition we mention these bounds. For this sake, we need to introduce the following concept defined in \cite{tian-2022+}. An \emph{independent total mutual-visibility set} of a graph $G$ is a set of vertices in $G$ that is both an independent set and a total mutual-visibility set of $G$. The cardinality of a largest independent total mutual-visibility set is the \emph{independent total
mutual-visibility number} of $G$, denoted by $\mu_{it}(G)$.

\begin{theorem}{\em \cite{tian-2022+}}
\label{th:cart-bounds}
If $G$ and $H$ are graphs with $n(G)\ge 2$, $n(H)\ge 2$, $\mu_{it}(G)\ge 1$, and $\mu_{it}(H)\ge 1$, then
$$\max\{\mu_{it}(H) \mu_t(G), \mu_{it}(G) \mu_t(H)\}\le \mu_t(G\Box H)\le \min\{n(G) \mu_t(H), n(H) \mu_t(G)\}.$$
\end{theorem}

In order to present a result which improves the upper bound mentioned above, we need to state the following lemma.

\begin{lemma}
\label{NotBypassInG-NOtInCartesian}
For any connected graph $G$ of order at least two,  $$\mathcal{P}(G)\times V(H)\subseteq \mathcal{P}(G\Box H).$$
\end{lemma}

\begin{proof}
If $x\in \mathcal{P}(G)$, then there exist two vertices  $g,g'\in V(G)\setminus \{x\}$ such that $N_G[g]\cap N_G[g']=\{x\}$. Hence, for any vertex $y\in V(H)$,
\begin{align*}
\{(x,y)\}&=N_G[g]\cap N_G[g']\times \{y\}
\\ & =(N_G[g]\times \{y\}\cup \{g\}\times N_H[y])\cap (N_G[g']\times \{y\}\cup \{g'\}\times N_H[y])\\
&=N_{G\Box H}[(g,y)]\cap N_{G\Box H}[(g',y)].
\end{align*}
Hence, $(x,y)\in \mathcal{P}(G\Box H)$, which implies that $\mathcal{P}(G)\times V(H)\subseteq \mathcal{P}(G\Box H)$.
\end{proof}

\begin{theorem}
\label{th:cart-upper-bound}
For any connected graphs $G$ and $H$, $$\mu_t(G\Box H)\le \min\{(\n(G)-|\mathcal{P}(G)|)\mu_t(H),(\n(H)-|\mathcal{P}(H)|)\mu_t(G)\}.$$
\end{theorem}

\begin{proof}
Let $X$ be a $\mu_t(G\Box H)$-set. By Lemma~\ref{NotBypassInG-NOtInCartesian} we deduce that, $(\mathcal{P}(G)\times V(H))\cap X = \varnothing$.
Hence,
\begin{align*}
  \mu_t(G\Box H) & = |X| = \sum_{u\in \mathcal{P}(G)} |(\{u\}\times V(H))\cap X| + \sum_{u\notin \mathcal{P}(G)} |(\{u\}\times V(H))\cap X| \\
                 & \le 0 + (\n(G)-|\mathcal{P}(G)|)\mu_t(H).
\end{align*}

Analogously we can prove that $\mu_t(G\Box H) \le (\n(H)-|\mathcal{P}(H)|)\mu_t(G)$. Therefore, the result follows.
\end{proof}

By Theorem~\ref{th:cart-upper-bound} and Proposition~\ref{pro:tmv-general-bounds-equality} (ii) we deduce the following bound.

\begin{theorem}
\label{th:cart-upper-bound-Simplicials}
Let $G$ and $H$ be two connected graphs. If $\mu_t(G) = |\mathcal{S}(G)|$ or $\mu_t(H) = |\mathcal{S}(H)|$, then  $\mu_t(G\Box H)\le \mu_t(G)\mu_t(H).$
\end{theorem}

As we will show below, the bound above is tight.

\begin{theorem}{\em \cite{tian-2022+}}
\label{th:cart-tmv-G-tree}
If $T$ is tree with $n(T)\ge 3$ and $H$ is a graph with $n(H)\ge 2$, then $\mu_t(T\Box H)= \mu_t(T)\mu_t(H)$.
\end{theorem}

We next show a result which includes Theorem \ref{th:cart-tmv-G-tree} as a particular case. Notice that for  any tree $T$ the only $\mu_t(T)$-set is the set of leaves (simplicial vertices) and clearly it is an independent set. Moreover, there exist numerous examples of graphs, where $\mu_t(G) = |\mathcal{S}(G)|$ and $\mathcal{S}(G)$ is an independent set, and among them we have, for instance,  the corona products $G\cong G^*\odot \overline{K_n}$, for any graph $G^*$, and any graph $G'$ obtained as follows. We begin with a Hamiltonian graph $H^*$ with Hamiltonian cycle $v_0v_1\dots v_{n(H^*)-1}v_0$ and $n(H^*)$ vertices $u_0,u_1,\dots u_{n(H^*)-1}$. Then to form $G'$ we join each vertex $u_i$ with $v_i$ and $v_{i+1}$, for $i\in \{0,1,...,n(H^*)-1\}$ (where the operations with the subscripts $i$ are expressed modulo $n(H^*)$). Notice that every vertex of $H^*$ belongs to  $\mathcal{P}(G')$.

\begin{theorem}
\label{th:tmv-G-independent}
If $\mathcal{S}(G)$ is an independent set of $G$ and $\mu_t(G) = |\mathcal{S}(G)|$, then $$\mu_t(G\Box H) = \mu_t(G)\mu_t(H).$$
\end{theorem}

\begin{proof}
The result is obtained by combining the upper bound given by Theorem~\ref{th:cart-upper-bound-Simplicials} and the lower bound given by Theorem~\ref{th:cart-bounds}.
%
%
%
%
%
%
%
\end{proof}

\begin{corollary}
If $T_{1}$,$\dots$,$T_{k}$ is a family of trees of order at least three, then
$$\mu_t(T_1\Box \cdots \Box T_k)=\prod_{i=1}^k\n_1(T_i).$$
\end{corollary}

We will now establish the following lemma, which will be one of our tools.

\begin{lemma}
\label{InC_4atMostTwo}
Let $x,x'$ be two adjacent vertices of a graph $G$ and let $y,y'$ be two adjacent vertices of a graph $H$. If $X$ is a $\mu_t(G\Box H)$-set and $(x,y)\in X$, then $(x',y')\not\in X$.
\end{lemma}

\begin{proof}
Suppose that $(x,y),(x',y')\in X$. Since the only shortest paths between $(x',y)$ and $(x,y')$ are  $(x',y),(x',y'),(x,y')$ and $(x',y),(x,y),(x,y')$,  the vertices $(x,y)$ and $(x',y')$ are not $X$-visible, which is a contradiction.
\end{proof}

The set $\mathcal{S}(G)$ of simplicial vertices of a graph $G$ can be partitioned into true twin equivalence classes where two vertices $g,g'\in \mathcal{S}(G)$ belong to the same class if and only if they are true twins, i.e., whenever $N_G[g]=N_G[g']$.

\begin{theorem}\label{GeneralizationCompleteTimesComplete}
Let $G$ be a graph and let $\{C_1,\dots, C_k\}$ be a partition of $\mathcal{S}(G)$ into true twin equivalence classes.  If $\mu_t(G)=|\mathcal{S}(G)|$, then for any integer $n\ge 2$
$$\mu_t(G\Box K_n)=\sum_{i=1}^k \max\{|C_i|,n\}.$$
\end{theorem}

\begin{proof}
Let $v\in V(K_n)$ be a fixed vertex of $K_n$ and let $u_i\in C_i$ be a representative vertex of the class $C_i$. We define the following set.
$$X=\left( \bigcup_{n\ge |C_i|}\{u_i\}\times V(K_n)\right) \bigcup \left( \bigcup_{n< |C_i|}C_i \times \{v\}\right).$$
We proceed to show that $X$ is a total mutual-visibility set of $G\Box K_n$.
To this end, we differentiate the following cases for two vertices $(g,h),(g',h')\in V(G)\times V(K_n)$.

\vspace{0,3cm}
\noindent Case 1. $g,g'\in C_i$ for some class $C_i$.
Since the subgraph of $G$ induced by  $C_i$ is complete, the subgraph of $G\Box H$ induced by $C_i\times V(K_n)$ is the Cartesian product of two complete graphs, and by the construction  of $X$ the subgraph induced by   $X_i=X\cap (C_i\times V(K_n))$ is also complete. Hence, it is readily seen that  $(g,h)$ and $(g',h')$ are $X_i$-visible, and so they are $X$-visible.

\vspace{0,3cm}
\noindent Case 2. $g\not \in  C_i$ for every class $ C_i$.
 Let $g=g_0,\dots, g_l=g'$  and  $h=h_0,\dots, h_r=h'$ be two shortest paths. If $g=g'$, then the shortest path $(g,h)=(g,h_0),\dots, (g,h_r)=(g',h')$ does not have vertices in $X$. Now, assume  $g\ne g'$.
Notice that $\mathcal{S}(G)$ is a $\mu_t(G)$-set and, by Proposition~\ref{pro:tmv-general-bounds-equality} (ii), $\mathcal{P}(G)=V(G)\setminus \mathcal{S}(G)$. Hence, the path $g=g_0,\dots, g_l=g'$ is $\mathcal{S}(G)$-visible, and   by the construction of $X$, the path
$$(g,h)=(g,h_0),\dots, (g,h_{r-1}),(g,h'),(g_1,h')\dots, (g_l,h')=(g',h')$$
is $X$-visible.

According to the two cases above, $X$ is a total mutual-visibility set of $G\Box K_n$, which implies that
$$\mu_t(G\Box K_n)\ge |X|=\displaystyle\sum_{i=1}^k \max\{|C_i|,n\}.$$

Now, let $W$ be a $\mu_t(G\Box K_n)$-set.
As mentioned above, $V(G) = \mathcal{S}(G)\cup\mathcal{P}(G)$ and, by Lemma~\ref{NotBypassInG-NOtInCartesian}, we know that $\mathcal{P}(G)\times V(K_n)\subseteq \mathcal{P}(G\Box K_n) $, which implies that
$W\cap (\mathcal{P}(G)\times V(K_n))=\varnothing$.
On the other side, since the subgraph of $G$ induced by any class $C_i$ is complete, by Lemma~\ref{InC_4atMostTwo} we can conclude that for any pair of different vertices $x,x'\in C_i$ and any pair of different vertices $y,y'\in V(K_n)$ we have that
$|W\cap \{(x,y),(x',y')\}|\le 1$ and $|W\cap \{(x,y'),(x',y)\}|\le 1$. Hence, $|W\cap C_i\times V(K_n)|\le \max\{|C_i|,n\}$ for every class $C_i$. Thus,
$$\mu_t(G\Box K_n)=\sum_{i=1}^k |W\cap( C_i\times V(K_n))|\le
 \displaystyle\sum_{i=1}^k \max\{|C_i|,n\}.$$
Therefore, the result follows.
\end{proof}

From Theorem~\ref{GeneralizationCompleteTimesComplete} we derive the following result, which was recently obtained in  \cite{tian-2022+}.

\begin{corollary}{\rm \cite{tian-2022+} }
If $m,n\ge 2$ are integers, then $\mu_t(K_m\Box K_n)=\max\{m,n\}$.
\end{corollary}

Based on the corollary above, one might think that the total mutual-visibility number of the Cartesian product of at least three complete graphs $K_{n_1}$, $K_{n_2}$, $\dots$, $K_{n_k}$ ($k\ge 3$) equals $\max\{n_1,n_2,\dots, n_k\}$. However, this seems to be far from the reality. To see this we consider the example $K_3\Box K_3\Box K_2$. Figure \ref{exch Cartesian} shows a total mutual-visibility set (in bold) of $K_3\Box K_3\Box K_2$ of cardinality $4>\max\{3,3,2\}$. This situation allows to think that finding $\mu_t(K_{n_1}\Box K_{n_2}\Box \cdots \Box K_{n_k})$ (in particular for Hamming graphs) is a challenging problem.

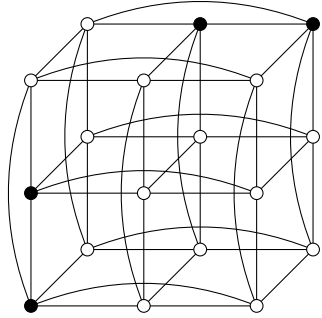
\begin{figure}[h]
\centering
\begin{tikzpicture}[transform shape, inner sep = .6mm]

\draw(0,0)--(3,0);
\draw(0,1.5)--(3,1.5);
\draw(0,3)--(3,3);
\draw(0,0)--(0,3);
\draw(1.5,0)--(1.5,3);
\draw(3,0)--(3,3);

\draw(0.75,0.75)--(3.75,0.75);
\draw(0.75,2.25)--(3.75,2.25);
\draw(0.75,3.75)--(3.75,3.75);
\draw(0.75,0.75)--(0.75,3.75);
\draw(2.25,0.75)--(2.25,3.75);
\draw(3.75,0.75)--(3.75,3.75);

\draw(0.75,0.75)--(0,0);
\draw(0.75,2.25)--(0,1.5);
\draw(0.75,3.75)--(0,3);
\draw(2.25,0.75)--(1.5,0);
\draw(2.25,2.25)--(1.5,1.5);
\draw(2.25,3.75)--(1.5,3);
\draw(3.75,0.75)--(3,0);
\draw(3.75,2.25)--(3,1.5);
\draw(3.75,3.75)--(3,3);

\draw (0,0) .. controls (-0.4,1) and (-0.4,2) .. (0,3);
\draw (1.5,0) .. controls (1.1,1) and (1.1,2) .. (1.5,3);
\draw (3,0) .. controls (2.6,1) and (2.6,2) .. (3,3);
\draw (0,0) .. controls (1,0.4) and (2,0.4) .. (3,0);
\draw (0,1.5) .. controls (1,1.9) and (2,1.9) .. (3,1.5);
\draw (0,3) .. controls (1,3.4) and (2,3.4) .. (3,3);

\draw (0.75,0.75) .. controls (0.35,1.75) and (0.35,2.75) .. (0.75,3.75);
\draw (2.25,0.75) .. controls (1.85,1.75) and (1.85,2.75) .. (2.25,3.75);
\draw (3.75,0.75) .. controls (3.35,1.75) and (3.35,2.75) .. (3.75,3.75);
\draw (0.75,0.75) .. controls (1.75,1.15) and (2.75,1.15) .. (3.75,0.75);
\draw (0.75,2.25) .. controls (1.75,2.65) and (2.75,2.65) .. (3.75,2.25);
\draw (0.75,3.75) .. controls (1.75,4.15) and (2.75,4.15) .. (3.75,3.75);

\node [draw=black, shape=circle, fill=black] (a1) at  (0,0) {} (0,0);
\node [draw=black, shape=circle, fill=black] (a2) at  (0,1.5) {} (0,0);
\node [draw=black, shape=circle, fill=white] (a3) at  (0,3) {} (0,0);
\node [draw=black, shape=circle, fill=white] (b1) at  (1.5,0) {} (0,0);
\node [draw=black, shape=circle, fill=white] (b2) at  (1.5,1.5) {} (0,0);
\node [draw=black, shape=circle, fill=white] (b3) at  (1.5,3) {} (0,0);
\node [draw=black, shape=circle, fill=white] (c1) at  (3,0) {} (0,0);
\node [draw=black, shape=circle, fill=white] (c2) at  (3,1.5) {} (0,0);
\node [draw=black, shape=circle, fill=white] (c3) at  (3,3) {} (0,0);

\node [draw=black, shape=circle, fill=white] (a'1) at  (0.75,0.75) {} (0,0);
\node [draw=black, shape=circle, fill=white] (a'2) at  (0.75,2.25) {} (0,0);
\node [draw=black, shape=circle, fill=white] (a'3) at  (0.75,3.75) {} (0,0);
\node [draw=black, shape=circle, fill=white] (b'1) at  (2.25,0.75) {} (0,0);
\node [draw=black, shape=circle, fill=white] (b'2) at  (2.25,2.25) {} (0,0);
\node [draw=black, shape=circle, fill=black] (b'3) at  (2.25,3.75) {} (0,0);
\node [draw=black, shape=circle, fill=white] (c'1) at  (3.75,0.75) {} (0,0);
\node [draw=black, shape=circle, fill=white] (c'2) at  (3.75,2.25) {} (0,0);
\node [draw=black, shape=circle, fill=black] (c'3) at  (3.75,3.75) {} (0,0);



\end{tikzpicture}
\caption{$K_3\Box K_3\Box K_2$ with total mutual-visibility set of cardinality $4$ in bold.}
\label{exch Cartesian}
\end{figure}

If  $v\in V(K_2)$ and $X$ is a $\mu_t(G)$-set, then $X\times \{v\}$  is a total mutual-visibility set of $G\Box K_2$. Hence, the lower bound $ \mu_t(G\Box K_2)\ge \mu_t(G)$ follows. Therefore, from Theorem~\ref{th:cart-bounds} we derive the following remark.
\begin{remark}
For any connected graph $G$,
$$\max\{2\mu_{it}(G),\mu_t(G)\}\le \mu_t(G\Box K_2)\le 2\mu_t(G).$$
Furthermore, if $\mu_t(G)=\mu_{it}(G)$, then $\mu_t(G\Box K_2)= 2\mu_t(G)$.
\end{remark}

From this remark we derive some open problems stated below.

\section{Concluding remarks}

In this article we have considered the total mutual-visibility number of graphs, by giving some tight bounds and closed formulae for this parameter. We have emphasized the investigation for the case of lexicographic and Cartesian product of graphs. Next, we propose some specific problems and possible research lines that can be taken as starting point for further researching on this topic.
\begin{itemize}
\item Investigate the behaviour of the total mutual-visibility number for the case of strong product graphs, direct product graphs,  generalized Sierpi\'{n}ski graphs and Hamming graphs (with emphasis on hypercubes).
\item Investigate how $\mu_t(G)$ is related to parameters other than $\diam(G)$ and $\gamma_c(G)$.
\item Characterize the vertices of a graph belonging to $\mathcal{C}(G)\setminus \mathcal{S}(G)$ and derive consequences of these characterizations.
\item Characterize the vertices of a graph belonging to $\mathcal{F}(G)\setminus \mathcal{P}(G)$ and derive consequences of these characterizations.
\item For connected graphs, whose complement is connected, study the existence of Nordhaus-Gaddum type relations.
\item Characterize the graphs $G$ with $\mu_t(G\Box K_2)=\mu_t(G).$
\item Characterize the graphs $G$ with $\mu_t(G\Box K_2)=2\mu_{it}(G).$
\item Characterize the graphs $G$ with $\mu_t(G\Box K_2)=2\mu_{t}(G).$
\item Study the complexity of computing $\mu_t(G)$.
\end{itemize}

\section*{Acknowledgements}

Dorota Kuziak was partially supported by the Spanish Ministry of Science and Innovation through the grant PID2019-105824GB-I00. Moreover, this investigation was
developed while the first author (Dorota Kuziak) was making a temporary stay at the Rovira i Virgili University supported by the program ``Ayudas para la recualificaci\'on del sistema universitario espa\~{n}ol para 2021-2023, en el marco del Real Decreto 289/2021, de 20 de abril de 2021''.

\section*{Declaration of interests}

The authors declare that they have no known competing financial interests or personal relationships that could have appeared to influence the work reported in this paper.

\section*{Data availability}

Our manuscript has no associated data.

\end{document}